\pdfoutput=1
\documentclass[reqno]{amsart}
\usepackage{graphicx} 
\usepackage{epsfig,subfigure}
\usepackage{a4wide}
\usepackage{dsfont}
\usepackage{float}
\usepackage[section]{placeins}
\usepackage{subcaption} 
\usepackage{amsmath,amsfonts,amssymb,latexsym,amsthm}
\usepackage{comment,hyperref,color} 
\usepackage{graphicx}
\usepackage{caption} 

\newtheorem{lemma}{Lemma}

\newtheorem{remark}{Remark}
\newtheorem{theorem}{Theorem}

\newtheorem{proposition}{Proposition}

\title{Anti-concentration applied to roots of randomized derivatives of polynomials}
\author{Andr\'e Galligo}
\address{Universit\'e C\^ote d'Azur,
Laboratoire de Math\'ematiques
"J.A. Dieudonn\'e", UMR CNRS 7351, Inria, 
Nice, France. }
\email{andre.galligo@univ-cotedazur.fr}
\author{Joseph Najnudel}
\address{University of Bristol, School of Mathematics, Bristol, United Kingdom.}
\email{joseph.najnudel@bristol.ac.uk}
\author{Truong Vu$^{*}$ }
\address{University of Illinois at Chicago, Department of Mathematics, Statistics, and Computer Science, Chicago, USA.}
\email{tvu25@uic.edu}
\thanks{$^*$ Corresponding author}
\date{}

\begin{document}

\begin{abstract}
     Let $(Z^{(n)}_k)_{1 \leq k \leq n}$ be a random set of points and let $\mu_n$  be its \emph{empirical measure}:
     $$\mu_n = \frac{1}{n} \sum_{k=1}^n \delta_{Z^{(n)}_k}. $$
 Let    
$$P_n(z) := (z - Z^{(n)}_1)\cdots (z - Z^{(n)}_n)\quad \text{and}\quad Q_n (z) := \sum_{k=1}^n \gamma^{(n)}_k
\prod_{1 \leq j \leq n, j \neq k} 
(z- Z^{(n)}_j),
$$
where $(\gamma^{(n)}_k)_{1 \leq k \leq n}$ are i.i.d. random variables
with Gamma distribution of parameter $\beta/2$, for some fixed $\beta > 0$. We prove that in the case where $\mu_n$ almost surely tends to $\mu$ when $n \rightarrow \infty$, the empirical measure of the complex zeros of the \emph{randomized derivative} $Q_n$ also converges almost surely to $\mu$ when $n$ tends to infinity. Furthermore, for $k = o(n / \log n)$, we obtain that the zeros of the $k-$th \emph{randomized derivative} of $P_n$ converge to the limiting measure $\mu$ in the same sense. We also derive the same conclusion for a variant of the randomized derivative related to the unit circle. 

\end{abstract}

\maketitle

\section{Introduction}
Suppose that $(Z^{(n)}_k)_{1 \leq k \leq n}$ is a random family of points in $\mathbb{C}$. For each $n \geq 1$, we define a random polynomial and its \emph{randomized derivative} as follows
\begin{align*}
    P_n(z)  = \prod_{j=1}^n (z - Z^{(n)}_j) \quad \text{and} \quad Q_n (z) = \sum_{k=1}^n \gamma^{(n)}_k
\prod_{1 \leq j \leq n, j \neq k} 
(z- Z^{(n)}_j),
\end{align*}
where $(\gamma^{(n)}_k)_{1 \leq k \leq n}$ are i.i.d. random variables
with Gamma distribution of parameter $\beta/2$, for some fixed $\beta > 0$. We define the empirical measure of $P_n$
$$\mu_n := \frac{1}{n} \sum_{k=1}^n \delta_{Z^{(n)}_k}.$$
Then, for $n \geq 2$, we define by $\nu_n$ the empirical 
measure of the zeros of $Q_n$, counted with multiplicity. The limit case $\beta =\infty$ corresponds to $\gamma^{(n)}_k = 1/n$, in which case 
$Q_n$ is $1/n$ times the derivative of $P_n$.  
It has been extensively researched how the distribution of zeros in random polynomials and the distribution of zeros in their (usual) derivatives relate to one another. In the case $(Z^{(n)}_k)_{1 \leq k \leq n}$ are i.i.d. random variables sampled from a probability measure $\mu$,  Pemantle and Rivin \cite{pemantle_rivin}  conjectured the empirical measure of $P_n$ and its derivative (in the usual sense) converge in distribution to the same limiting measure $\mu$.
\subsection{Previous work}
When $\mu$ is supported on the unit circle, Subramanian \cite{subramanian} demonstrated the validity of the Pemantle-Rivin conjecture. The conjecture of Pemantle and Rivin for \emph{all} probability measures $\mu$ was verified by Kabluchko in a seminal work \cite{Kabluchko2015}. Extending Kabluchko's work, Byun, Lee, and Reddy \cite{byun-lee-reddy} settled that for each fixed $k$ derivatives. With an eye toward extremely high derivatives, Steinerberger \cite{steinerberger2019nonlocal}, O'Rourke and Steinerberger \cite{o2021nonlocal} propose that the random measure $\mu_n^{\lfloor t n\rfloor}$ converges to a deterministic measure $\mu_t$ for any fixed $t \in [0,1]$.  O'Rourke and Steinerberger postulate that a specific partial differential equation is satisfied by the logarithmic potentials of the limiting measure $(\mu_t)_t$ when the underlying measure $\mu$ is radial. When $\mu$ is not assumed radial, the work of the first author \cite{Galligo2022} considered a specific
system of two PDEs to
describe the motion.  Analysts \cite{alazard2022dynamics,kiselev2022flow} have examined this partial differential equation, and Hoskins and Steinerberger \cite{hoskins2022semicircle} (see also \cite{hoskins2021dynamics,kabluchko2021repeated}) have proven O'Rourke and Steinerberger's conjecture in the special case where $\mu$ has real support. According to O'Rourke and Steinerberger's prediction, the limiting measure in the $t = o(1)$ situation should equal the underlying measure $\mu$. Michelen and the third author \cite{michelen2022zeros} recently demonstrated that the same holds if the number of derivatives rises in $n$ a little bit less slowly than logarithmically  $k \leq \frac{\log n}{5 \log \log n}$.  With the majority of the new components coming in to handle an anti-concentration estimate, the works \cite{byun-lee-reddy} and \cite{michelen2022zeros} on convergence of higher derivatives employ the same general strategy as Kabluchko's original proof \cite{Kabluchko2015}. According to the conjectures of Cheung-Ng-Yam \cite{cheung2014critical} and Kabluchko (see \cite{michelen2022zeros}), the empirical measure corresponding to the zeros of the first derivative of the random polynomial should actually weakly converge to $\mu$ \emph{almost surely} rather than merely in probability.  Angst, Malicet, and Poly \cite{angst2024almost} proved this conjecture recently for all probability measures $\mu$.  Additionally, Angst, Malicet, and Poly \cite{angst2024almost} conjectured that for every fixed $k\in\mathbb{N}$, the almost-sure convergence of the empirical measure corresponding to the $k-$th derivatives should hold. Very recently, in \cite{michelen2023almost}, Michelen and the third author confirmed their conjecture. The following works and their references provide additional information on this model and related ones: \cite{alazard2022dynamics,byun-lee-reddy, cheung2014critical,cheung2015higher, Galligo2022, hoskins2021dynamics, hoskins2022semicircle,Kabluchko2015,kabluchko2021repeated, kiselev2022flow,o2021nonlocal,hall2023roots,steinerberger2019nonlocal,steinerberger2023free,subramanian}.

\subsection{Our project}
In this paper, we study the relation between the distribution of zeros of random polynomials and the distribution of zeros of its \emph{randomized derivatives} with more general setting for initial random points $(Z^{(n)}_k)_{1 \leq k \leq n}$. This randomization is motivated by the following points: 
\begin{itemize}
\item If the polynomial $P_n$ has distinct roots $(Z^{(n)}_k)_{1 \leq k \leq n}$, then the zeros of $P'_n$ are given by the following equation in $z$: 
$$\sum_{k=1}^n \frac{1}{z - Z^{(n)}_k} = 0.$$
\item If $(Z^{(n)}_k)_{1 \leq k \leq n}$ has a limiting empirical distribution $\mu$, we can expect 
that approximately: 
$$\frac{1}{n} \sum_{k=1}^n \frac{1}{z - Z^{(n)}_k}  \simeq \int_{\mathbb{C}} \frac{d \mu(u)}{z - u}.$$
\item For values of $z$ such that the last integral is not close to zero, the left-hand side can vanish only if the approximation is not valid. 
\item This typically occurs if $z$ is close to one of the roots $Z^{(n)}_k$ of $P_n$: if it is at distance 
of order $1/n$, then the corresponding term 
$ \frac{1}{n(z - Z^{(n)}_k)}$ has order of magnitude $1$, and then can compensate the value of the integral with respect to $\mu$. 
\item It is then natural to expect that most of the roots of $P'_n$ are at distance of order $1/n$ from a root of $P_n$, which  suggest that the empirical measure of the roots of $P'_n$ is very close to the empirical measure of the roots of $P_n$. More precise results related to this intuition are given in \cite{SeanO’Rourke_NoahWilliams2020}.
\item These consideration can be extended to the case where we consider the roots of the randomized derivative $Q_n$, giving the equation: 
$$\sum_{k=1}^n \frac{\gamma_k^{(n)}}{z - Z^{(n)}_k} = 0.$$
\item Hence, it is natural to expect similar comparison between the roots of $P_n$ and $Q_n$: most of the roots of $Q_n$ are expected to be very close to roots of $P_n$.
\end{itemize}
Another motivation of introducing this randomization comes from random matrix theory. Indeed, for $\beta = 1$ or $\beta =2 $, the zeros of the randomized derivative correspond to the eigenvalues of the $(n-1) \times (n-1)$ top-left minor of the matrix
$UDU^{-1}$, $D$ being diagonal with diagonal entries $(Z^{(n)}_k)_{1 \leq k \leq n}$, $U$ being independent, Haar-distributed on the orthogonal group $O(n)$ (for $\beta =1$) or the unitary group $U(n)$ (for $\beta = 2$): see Subsection \ref{ssection:e-value_minors} below for more detail. Randomized derivative defined as above is also related to $\beta$-ensembles for general $\beta > 0$: see \cite{AssiotisNajnudel}. 
The connection with random matrix theory is the reason why we choose Gamma-distributed weights $\gamma_k^{(n)}$: more general distributions can be considered, as discussed in Remark \ref{generaldistribution} below. 

Proofs that iterated derivatives of a polynomial preserves the limiting empirical distribution of the roots use anti-concentration estimates, in order to show that suitable random variables have small probability to be close to some given values. In the present setting, introducing randomness smooth the distribution of the random variables which are involved, which implies stronger anti-concentration results. From that, we are able to prove results which are stronger than what has been proven in the case of non-randomized derivative. 

Our main result is the following: we show that, for $k=o(\frac{n}{\log n})$ (significantly larger than what it obtained in the work of Michelen and the third author \cite{michelen2022zeros}), the empirical measures corresponding the random polynomial and its $k-$th \emph{randomized derivatives} converge to the same limiting measure. The main ingredient is using anti-concentration properties obtained through the randomness involved in the \emph{randomized derivatives}.  In the present article, we prove that in the case where $\mu_n$ almost surely tends to $\mu$ when $n \rightarrow \infty$, the empirical measure of the complex zeros of the \emph{randomized derivative} $Q_n$ also converges almost surely to $\mu$ when $n$ tends to infinity. Furthermore, for $k = o(n / \log n)$, we obtain that the zeros of the $k-$th \emph{randomized derivative} of $P_n$ converge to the limiting measure $\mu$ in the same sense. We also derive the same conclusion for a variant of the randomized derivative related to the unit circle. 

\subsection{Organisation of the paper}
The remainder of the paper is organized as follows. In Section \ref{sec:rel_RMT}, we discuss the relation with random matrices theory. In Section \ref{sec:mainresults}, we state precisely our main results. In Section \ref{sec:usual_randomized}, we present the proofs of Theorem \ref{thm:a.e.converg1stderivative}, Theorem \ref{thm:a.e.converg_n_logn_derivative_prop3} and Theorem \ref{thm:a.e.converg_n_logn_derivative_prop5}. The proofs of Theorem \ref{thm:prop7}, Theorem \ref{thm:prop8} and Theorem \ref{thm:a.e.converg_n_logn_derivative_prop9} are similar to Theorem \ref{thm:a.e.converg1stderivative}, \ref{thm:a.e.converg_n_logn_derivative_prop3} and Theorem \ref{thm:a.e.converg_n_logn_derivative_prop5}, and build on similar technical results as those presented in Section \ref{sec:usual_randomized}. The details will be presented in Section \ref{sec:variation_randomized}.

\section{Relation with Random matrix theory}\label{sec:rel_RMT}
\subsection{Eigenvalues of minors}\label{ssection:e-value_minors}
  The cases $\beta =1$ and $\beta =2$ can be interpreted in terms of random matrices. Indeed, let 
 us consider the matrix 
$$M = U^* \operatorname{Diag}(\lambda_1, \dots, \lambda_n) U$$
where $U \in U(n)$, 
$\lambda_1, \dots, \lambda_n \in \mathbb{C}$. 
Then, the characteristic polynomial of the top-left $(n-1) \times (n-1)$ minor of $M$ is the cofactor of index $nn$ of the matrix 
$$U^* \operatorname{Diag}(\lambda_1 -z, \dots, \lambda_n-z ) U  \in \mathcal{M}_n (\mathbb{C}(z)),$$
which is equal to 
$$\operatorname{\det} (U^* \operatorname{Diag}(\lambda_1 -z, \dots, \lambda_n-z ) U ) \left( (U^* \operatorname{Diag}(\lambda_1 -z, \dots, \lambda_n-z ) U )^{-1} \right)_{nn}$$
and then to 
$$\prod_{j=1}^n (\lambda_n - z) \, (U^* \operatorname{Diag}((\lambda_1 -z)^{-1}, \dots, (\lambda_n-z)^{-1} ) U )_{nn}
$$
or 
$$\prod_{j=1}^n (\lambda_j - z)  
\sum_{j=1}^n U^*_{nj} (\lambda_j - z)^{-1} U_{jn}= \sum_{j=1}^n |U_{jn}|^2 \prod_{k \neq j} (\lambda_k- z).$$
 This gives the randomization setting above for $\beta = 2$, since for $U$ uniform on $U(n)$, $(|U_{jn}|^2)_{1 \leq j \leq n}$ is 
 Dirichlet distributed with all parameters equal to $1$, and then has the same joint distribution as 
 $$\left(\frac{\gamma^{(n)}_j}{\sum_{k=1}^n \gamma^{(n)}_k}\right)_{1 \leq j \leq n}.$$
The case $\beta =1$ is similarly obtained for $U$ uniform on $O(n)$, in which case $(|U_{jn}|^2)_{1 \leq j \leq n}$ is 
 Dirichlet distributed with all parameters equal to $1/2$. 
\subsection{Eigenvalues of products of reflexions}\label{ssec:con}
 The circular variant of the randomized derivative
$$V_n(z) = P_n(z)  \sum_{k=1}^n  \frac{Z^{(n)}_k + z}{Z^{(n)}_k - z} \gamma^{(n)}_k $$
is related to the evolution of the eigenvalues of a unitary matrix when it is multiplied by a complex reflection. We refer to the beginning of Najnudel-Virag paper \cite{najnudel2021bead} for a discussion on this setting. Under some conditions which are specified in that paper, 
the new eigenvalues are obtained from the initial eigenvalues
$(Z^{(n)}_k)_{1 \leq k \leq n}$ by solving the equation
$$\sum_{k=1}^n 
\frac{Z^{(n)}_k + z}{Z^{(n)}_k - z} \rho^{(n)}_k
= \frac{1 + \eta}{1-\eta}$$
for some complex number $\eta $ of modulus one, $(\rho^{(n)}_j)_{1 \leq j \leq n}$ being Dirichlet distributed with all parameters equal to $\beta/2$ for $\beta =2$. 
If we generalize the setting to all $\beta > 0$, and if we consider the particular case $\eta = -1$, we get the equation 
$$\sum_{k=1}^n 
\frac{Z^{(n)}_k + z}{Z^{(n)}_k - z} \gamma^{(n)}_k = 0,$$
i.e. $V_n(z) = 0$, if the Dirichlet distribution is realized by taking 
$$ \rho^{(n)}_k = \frac{\gamma^{(n)}_k }{ \sum_{j=1}^n \gamma^{(n)}_j }. $$
Notice that for $(Z^{(n)}_j)_{1 \leq j \leq n}$ distinct points on the unit circle, $$\sum_{k=1}^n  \frac{Z^{(n)}_k + z}{Z^{(n)}_k - z} \gamma^{(n)}_k $$
 is purely imaginary when $z$ is on the unit circle, and goes from $- i\infty$ to $i \infty$ when $z$ spans the arc between two consecutive poles: one deduces that all solutions of the equation are on the unit circle, and that they interlace between the points $(Z^{(n)}_j)_{1 \leq j \leq n}$. 
The limiting case $\beta = \infty$ gives 
$$\frac{V_n(z)}{P_n(z)} = \frac{1}{n} \sum_{k=1}^n 
\frac{Z^{(n)}_k + z}{Z^{(n)}_k - z} = 
 \sum_{k=1}^n  \frac{2z/n}{Z^{(n)}_k - z} + \sum_{k=1}^n  \frac{(Z^{(n)}_k - z)/n}{Z^{(n)}_k - z}
 = - \frac{2z}{n} \frac{P_n'(z)}{P_n(z)} + 1,
$$
i.e. 
$$V_n(z) = -\frac{2}{n} \left(z P'_n(z) - \frac{n}{2} P_n(z) \right).$$
\section{Statements of our results}\label{sec:mainresults}
We first prove the following
result for the  first order randomized derivative:

\begin{theorem}\label{thm:a.e.converg1stderivative}
If $(\mu_n)_{n \geq 1}$ converges almost surely to a deterministic measure $\mu$ on $\mathbb{C}$, then 
$(\nu_n)_{n \geq 2}$ converges almost surely to $\mu$. 
\end{theorem}
Our key ingredient for the proof of Theorem \ref{thm:a.e.converg1stderivative} is to handle an anti-concentration problem which is the content of Proposition \ref{prop1} in Section \ref{sec:usual_randomized}.\\

We can now iterate the construction, similarly as in the case where higher order derivatives are considered. 
Fix $\beta > 0$, and for $n \geq 2$, we choose an integer $m(n)$, $1 \leq m(n) \leq n-1$, representing the number of iterations which are performed. We fix a random set of points $(Z_k^{(n,0)})_{1 \leq k \leq n}$
and an independent set of i.i.d. random variables $(\gamma_k^{(n,j)})_{1 \leq j \leq m(n), 1 \leq k \leq n+1 - j}$, with Gamma distribution of parameter $\beta/2$. Then, for $1 \leq j \leq m(n)$, we inductively define the set of points $(Z_k^{(n,j)})_{1 \leq k \leq n-j}$, as the set of zeros, counted with multiplicity, of the polynomial 
$$z \mapsto \sum_{k=1}^{n+1-j} \gamma^{(n,j)}_k   \prod_{1 \leq \ell \leq n+1-j, \ell \neq k} (z - Z_{\ell}^{(n,j-1)}).$$ 
For $n \geq 2$ and $0 \leq j \leq m(n)$, let
$$\mu_{n,j} := \frac{1}{n-j} \sum_{k=1}^{n-j} \delta_{Z_k^{(n,j)}}.$$
We obtain the following results for higher randomized derivatives:
\begin{theorem} \label{thm:a.e.converg_n_logn_derivative_prop3}
We assume that $\mu_{n,0}$ converges almost surely to a deterministic measure $\mu$ on $\mathbb{C}$, 
that 
$$\frac{m(n)}{n/\log n} \underset{n \rightarrow \infty}{\longrightarrow} 0$$
and that for some $B \geq 0$, 
$$ \sup_{n \geq 2} n^{-B} \sup_{ 0 \leq j \leq m(n)-1} \mathbb{E} \left[ \min_{1 \leq k \leq n-j} |Z_k^{(n,j)}| \right] < \infty. 
$$
Then, $(\mu_{n, m(n)})_{n \geq 2}$ converges almost surely to $\mu$. 
\end{theorem}

The last condition of Theorem \ref{thm:a.e.converg_n_logn_derivative_prop3}
does not look very easy to check. A sufficient condition is given as follows: 
\begin{lemma}\label{lem:prop4}
If for some $B > 0$
$$ \sup_{n \geq 2} n^{-B}   \mathbb{E} \left[ \max_{1 \leq k \leq n} |Z_k^{(n,0)}| \right] < \infty, 
$$
then last condition of Theorem \ref{thm:a.e.converg_n_logn_derivative_prop3}
is satisfied. 
 
 \end{lemma}

\begin{theorem}\label{thm:a.e.converg_n_logn_derivative_prop5}
Let $\mu$ be a (deterministic) probability measure on $\mathbb{C}$ with at least polynomially decaying tail, i.e.  there exists $c > 0$ with 
$$\sup_{R \geq 1} R^c \mu( \{z \in \mathbb{C}, 
|z| > R \}) < \infty. $$
Let $(Z^{(n,0)}_k)_{1 \leq k \leq n}$ be i.i.d. random variables with distribution $\mu$. Then, for $1 \leq m(n) \leq n-1$ such that 
$$\frac{m(n)}{n/\log n} \underset{n \rightarrow \infty}{\longrightarrow} 0,$$
$(\mu_{n,  m(n)})_{n \geq 2}$ converges almost surely to $\mu$. 

 \end{theorem}
The following figures illustrate for the distribution of the roots of $P$ and its randomized derivatives with $100$ initial points.
\begin{figure}[H] 
  \centering
\includegraphics[width=0.3455\linewidth]{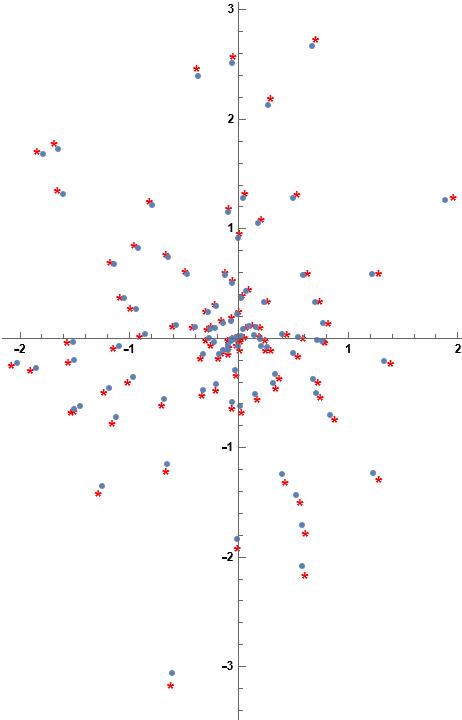} 
    \label{fig:0_3der} \caption{The behavior of the roots of $P_n$ and the third randomized derivative polynomials.}   
    \end{figure}
\begin{figure}[H]
  \centering
\includegraphics[width=0.346\linewidth]{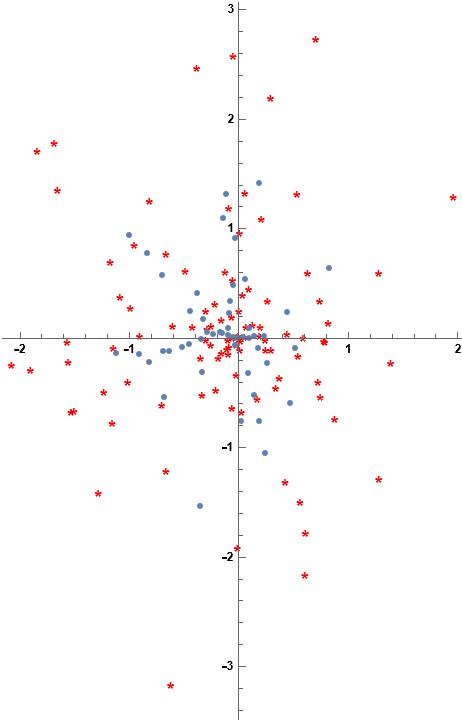}
   \label{fig:0_50der}
   
\caption{The behavior of the roots of $P_n$ and the $50$th randomized derivative polynomials.}
\end{figure}
From the above figure, we can see that the behavior of the roots of $P_n$ and $k$-th randomized derivative polynomials are similar for $k$ small (Figure $1$) and very different for $k$ large (Figure $2$).

  \begin{remark}
We can ask what happens if we consider the limiting cases $\beta =0$ and $\beta =\infty$.  The case $\beta=\infty$ corresponds to the successive derivatives of a given polynomial: this case is still open under the assumption $m(n) = o(n/\log n)$ (see \cite{michelen2022zeros},\cite{michelen2023almost}).  The case $\beta = 0$ consists, at each step, of choosing one of the points uniformly at random, and removing it. If we do $o(n)$ iterations, the limiting empirical measure is not changed by this procedure: our result remains valid. 
  \end{remark}
Next, we define the circular variant of the "randomized derivative" of $P_n$, given, with the same notation as at the beginning, by 
$$V_{n}(z) = P_{n}(z)  \sum_{k=1}^n  \frac{Z^{(n)}_k + z}{Z^{(n)}_k - z} \gamma^{(n)}_k. $$
For $\beta=\infty$, $V_n(z)$ is proportional to $ zP_n'(z)-\frac{n}{2}P_n(z)$ which corresponds to the operator studied by Kabluchko in \cite{kabluchko2021repeated} (see Subsection \ref{ssec:con}). For $n \geq 2$, let $\tilde\nu_n$ be the empirical 
measure of the zeros of $V_n$, counted with multiplicity. We can iterate the construction, similarly as in the case where higher order derivatives are considered. 
We fix $\beta > 0$, and for $n \geq 2$, we choose an integer $m(n) \geq 1$, representing the number of iterations which are performed. We fix a random set of points $(Z_k^{(n,0)})_{1 \leq k \leq n}$
and an independent set of i.i.d. random variables $(\gamma_k^{(n,j)})_{1 \leq j \leq m(n), 1 \leq k \leq n}$, with Gamma distribution of parameter $\beta/2$. Then, for $1 \leq j \leq m(n)$, we inductively define the set of points $(Z_k^{(n,j)})_{1 \leq k \leq n}$, as the set of zeros, counted with multiplicity, of the polynomial 
$$z \mapsto \sum_{k=1}^{n} \gamma^{(n,j)}_k (Z_k^{(n,j-1)}+z)  \prod_{1 \leq \ell \leq n, \ell \neq k} (z - Z_{\ell}^{(n,j-1)}).$$

We observe that the coefficients $(\gamma^{(n,j)}_k)_{1 \leq k \leq n}$ are independent of 
$ (Z_{\ell}^{(n,j-1)})_{1 \leq \ell \leq n}$, and then we are in the same situation as above when we go
from the points $ (Z_{\ell}^{(n,j-1)})_{1 \leq \ell \leq n}$ to the points $(Z_k^{(n,j)})_{1 \leq k \leq n}$. We can then apply Proposition \ref{prop6} below.

For $n \geq 2$ and $0 \leq j \leq m(n)$, let
$$\tilde\mu_{n,j} := \frac{1}{n} \sum_{k=1}^{n} \delta_{Z_k^{(n,j)}}.$$

Similarly as for the randomized derivatives model, we get the following results for its circular variant. 
\begin{theorem}\label{thm:prop7}
If $(\mu_n)_{n \geq 1}$ converges almost surely to a deterministic measure $\mu$ on $\mathbb{C}$, then 
$(\tilde\nu_n)_{n \geq 2}$ converges almost surely to $\mu$. 
\end{theorem}

\begin{theorem} \label{thm:prop8}
We assume that $\tilde\mu_{n,0}$ converges almost surely to a deterministic measure $\mu$ on $\mathbb{C}$, 
that 
$$\frac{m(n)}{n/\log n} \underset{n \rightarrow \infty}{\longrightarrow} 0$$
and that for some $B \geq 0$, 
$$ \sup_{n \geq 2} n^{-B} \sup_{ 0 \leq j \leq m(n)-1} \mathbb{E} \left[ \min_{1 \leq k \leq n} |Z_k^{(n,j)}| \right] < \infty. 
$$

Then, $(\tilde\mu_{n, m(n)})_{n \geq 2}$ converges almost surely to $\mu$. 
\end{theorem}

\begin{theorem}\label{thm:a.e.converg_n_logn_derivative_prop9}
Let $\mu$ be a (deterministic) probability measure on $\mathbb{C}$ with at least polynomially decaying tail, i.e.  there exists $c > 0$ with 
$$\sup_{R \geq 1} R^c \mu( \{z \in \mathbb{C}, 
|z| > R \}) < \infty. $$
Let $(Z^{(n,0)}_k)_{1 \leq k \leq n}$ be i.i.d. random variables with distribution $\mu$. Then, for $1 \leq m(n) \leq n-1$ such that 
$$\frac{m(n)}{n/\log n} \underset{n \rightarrow \infty}{\longrightarrow} 0,$$
$(\tilde\mu_{n,  m(n)})_{n \geq 2}$ converges almost surely to $\mu$. 

 \end{theorem}

The following figures illustrate for the distribution of the roots of $P$ and its circular variant randomized derivatives with $60$ initial points.
\begin{figure}[H] 
  \centering
\includegraphics[width=0.7\linewidth]{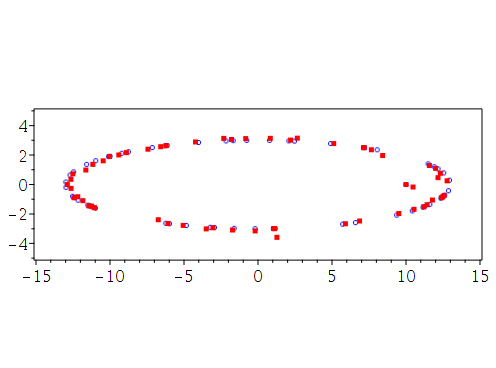} 
    \label{fig:rootV}
    \caption{The behavior of the roots of $P_n$ and the variant of the first randomized derivative polynomial.}
  \end{figure}  
\begin{figure}[H]
  \centering
\includegraphics[width=0.5\linewidth]{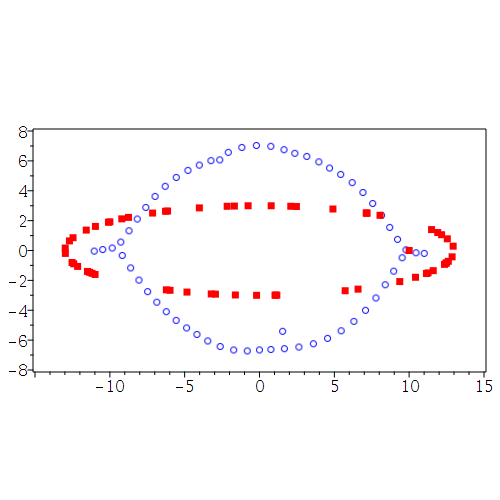}
   \label{fig:root_20variant_Rderivative_P}
   \caption{The behavior of the roots of $P_n$ and the variant of the $20$th randomized derivative polynomial.}
\end{figure}

From the above figure, we can see that the behavior of the roots of $P_n$ and the $k$-{\text{th}} circular variant of the randomized derivative polynomials are similar for $k$ small (Figure $3$) and very different for $k$ large (Figure $4$).

\section{The usual randomized derivation model}\label{sec:usual_randomized}
The proof of the theorems is based on the following proposition: 
\begin{proposition} \label{prop1}
For all $\beta > 0$, $A > 0$, and for all smooth functions $\varphi : \mathbb{C} \mapsto \mathbb{R}$, 
with compact support, there exist 
$C, K > 0$, depending only on 
$\beta$, $A$, $\varphi$, such that 
for all $n \geq 2$, 
$$ \mathbb{P} \left( \left|\int_{\mathbb{C}} \varphi d \nu_n
- \int_{\mathbb{C}}\varphi d \mu_n
\right| 
\geq \frac{K \log n }{n} \right) 
\leq C n^{-A} \left(1 + \mathbb{E} \left[
\min_{1 \leq k \leq n} |Z_k^{(n)}| \right]
\right). 
$$
\end{proposition}

\begin{proof}[Proof of Proposition \ref{prop1}]
By first conditioning on $(Z^{(n)}_k)_{1 \leq k \leq n}$, which does not change the distribution of 
$(\gamma^{(n)}_k)_{1 \leq k \leq n}$ since these families of variables are assumed to be independent, 
we can assume that $(Z^{(n)}_k)_{1 \leq k \leq n}$ is deterministic. 
Moreover, the ordering of
$(Z^{(n)}_k)_{1 \leq k \leq n}$ does not change the distribution of the random measure $\nu_n$, since $(\gamma^{(n)}_k)_{1 \leq k \leq n}$
are i.i.d.: we can assume that 
$|Z^{(n)}_1| \leq |Z^{(n)}_k| $
for $2 \leq k \leq n$. It is then enough to prove, for 
deterministic $(Z^{(n)}_k)_{1 \leq k \leq n}$, 
$$ \mathbb{P} \left( \left|\int_{\mathbb{C}} \varphi d \nu_n
- \int_{\mathbb{C}}\varphi d \mu_n \right| 
\geq \frac{K \log n }{n} \right) 
\leq C n^{-A} \left(1 +
 |Z_1^{(n)}| 
\right). 
$$

 We have 
$$n \int_{\mathbb{C}} \varphi d \mu_n 
= \frac{1}{2 \pi} \int_{\mathbb{C}} \log |P_n | \, \Delta \varphi \, d m (z)$$
and 
$$(n-1) \int_{\mathbb{C}} \varphi d \nu_n 
= \frac{1}{2 \pi} \int_{\mathbb{C}} \log |Q_n | \, \Delta \varphi \, d m (z)$$
where $m$ is the Lebesgue measure on $\mathbb{C}$. 
Hence, 
$$\int_{\mathbb{C}} \varphi d \nu_n
- \int_{\mathbb{C}}\varphi d \mu_n 
=  \frac{1}{n} \int_{\mathbb{C}} \varphi d \nu_n +  \frac{1}{2 \pi n} 
\int_{\mathbb{C}} \log \left| \sum_{k=1}^n 
\frac{ \gamma_k^{(n)}}{ z - Z^{(n)}_k} \right| \, \Delta \varphi \, d m (z).$$
The first term of the right-hand side is 
bounded by $1/n$ times the maximum of $|\varphi|$, and then by 
$K_1 \log n/ n$ for some $K_1 > 0$ depending only on $\varphi$. 
The second term is bounded by 
the maximum of $|\Delta \varphi|$, 
multiplied by 
$$\frac{1}{2 \pi n} 
\int_{\mathbb{D}_R} 
\log_+ \left| \sum_{k=1}^n 
\frac{ \gamma_k^{(n)}}{ z - Z^{(n)}_k} \right| 
d m(z) 
+ \frac{1}{2 \pi n} 
\int_{\mathbb{D}_R} 
\log_- \left| \sum_{k=1}^n 
\frac{ \gamma_k^{(n)}}{ z - Z^{(n)}_k} \right|  d m(z), 
$$
where $R \geq 1$ is chosen in such a way that 
$\varphi$ is supported on 
$$\mathbb{D}_R := \{ z \in \mathbb{C}, |z| \leq R\}. $$

It is then enough to show that there exists $C, K_2, K_3 > 0$, depending only on $\beta$, $A$ and $R$, such that

$$\mathbb{P} \left(
\int_{\mathbb{D}_R}
\log_+ \left| \sum_{k=1}^n 
\frac{ \gamma_k^{(n)}}{ z - Z^{(n)}_k} \right| 
d m(z)  > K_2 \log n \right) 
\leq C  n^{-A} (1 + |Z^{(n)}_1|) 
$$
and 
$$\mathbb{P} \left(
\int_{\mathbb{D}_R}
\log_- \left| \sum_{k=1}^n 
\frac{ \gamma_k^{(n)}}{ z - Z^{(n)}_k} \right| 
d m(z)  > K_3 \log n \right) 
\leq C  n^{-A} (1 + |Z^{(n)}_1|) 
$$

For the integral involving $\log_+$,
we first observe the following: 
for $u \geq 1$, and for $C_{\beta} > 0$ depending only on $\beta$, 
$$\mathbb{P} \left( \gamma_k^{(n)} \geq u
\right) = \frac{1}{ \Gamma(\beta/2)} 
\int_{u}^{\infty} e^{-t} t^{\beta/2-1} dt
\leq C_{\beta} 
\int_u^{\infty} e^{-t/2} dt 
\leq 2 C_{\beta} e^{-u/2}. 
$$
Since this probability, in particular, decreases faster than the $u^{-c}$ for some universal constant $c > 0$, we can write, using a union bound in order to deal with the maximum, 
$$\mathbb{P} \left( \underset{1 \leq k \leq n}{\max} \gamma^{(n)}_k 
\geq n^{(A+1)/c} \right) \leq C n^{-A} $$
for some $C$ depending only on $\beta$ and $A$. 
It is then enough to show that under the complementary event $\gamma_k^{(n)} \leq n^{(A+1)/c}$, we have an equality of the form 
$$\int_{\mathbb{D}_R}
\log_+ \left| \sum_{k=1}^n 
\frac{ \gamma_k^{(n)}}{ z - Z^{(n)}_k} \right| dm(z) \leq K_2 \log n.  $$
Now, for any integer $m \geq 0$, the term  
$\log_+$ just above
can be larger than 
$(10 +  (A+1)/c ) \log n + m$ only if one of the 
terms $\gamma_k^{(n)}/| z - Z^{(n)}_k|$
is larger than or equal to $n^{9+ (A+1)/c } e^m$, 
which restricts $z$ to a disc of radius 
$ e^{-m} n^{-9- (A+1)/c}\gamma_k^{(n)}
\leq e^{-m} n^{-9}$ if we assume  
$\gamma_k^{(n)} \leq n^{(A+1)/c}$. 
The area of the set of $z$ for which   $\log_+$ is larger than $(10 + (A+1)/c) \log n + m$ is then, by adding the areas of the $n$ small discs, dominated by 
$e^{-2m} n^{-17}$. 

Bounding $\log_+$ by $(10+(A+1)/c) \log n$ 
if it is smaller than $(10+(A+1)/c) \log n$, and by 
 $(10 + (A+1)/c) \log n + m +1$ if it is between 
 $(10 + (A+1)/c) \log n + m$ and $(10 + (A+1)/c) \log n + m+1$,
 we deduce that 
\begin{align*}
&  
\int_{\mathbb{D}_R}
\log_+ \left| \sum_{k=1}^n 
\frac{ \gamma_k^{(n)}}{ z - Z^{(n)}_k} \right| 
d m(z) 
\\ & \leq  
( \pi R^2) ((10 + (A+1)/c) \log n) 
+  \sum_{m=0}^{\infty}
 \mathcal{O} (e^{-2m} n^{-17}) 
 ((10 + (A+1)/c) \log n + m + 1)
\leq  K_2 \log n 
\end{align*}
where $K_2 > 0$ depends only on $A$ and $R$. 

It remains to prove the bound on the integral involving $\log_-$. 
If we replace $\log_-$ by its minimum with $K_4 \log n$
for some $K_4 > 0$ depending only on $\beta$, $A$, $R$, we get an integral on $\mathbb{D}_{R} $
bounded by $\pi R^2 K_4 \log n$, which shows that it is enough to check 

$$\mathbb{P} \left(
\int_{\mathbb{D}_R}
\left(\log_- \left| \sum_{k=1}^n 
\frac{ \gamma_k^{(n)}}{ z - Z^{(n)}_k} \right|  - K_4 \log n \right)_+
d m(z)  \geq  10 \log n  \right) 
\leq C  n^{-A} (1 + |Z^{(n)}_1|) 
$$
for $C > 0$ depending only on $\beta, A, R$. 
Since $10 \log n \geq 1$, by Markov inequality, it is enough to show 
$$ \int_{\mathbb{D}_R}
\mathbb{E} \left[ \left(\log_- \left| \sum_{k=1}^n 
\frac{ \gamma_k^{(n)}}{ z - Z^{(n)}_k} \right| 
- K_{4} \log n  \right)_+ \right] dm(z) 
 \leq C n^{-A} ( 1 + |Z^{(n)}_1|) .$$
The last expectation if equal to 
$$\int_{K_{4} \log n}^{\infty} 
\mathbb{P} \left( \left| \sum_{k=1}^n 
\frac{ \gamma_k^{(n)}}{ z - Z^{(n)}_k}  \right|
\leq e^{-t} \right)  dt.
$$
We recall that we assume $(Z^{(n)}_k)_{
1 \leq k \leq n}$ deterministic. 
If we condition on $(\gamma^{(n)}_k)_{2 \leq k \leq n}$, 
the conditional distribution of $\gamma^{(n)}_1$ remains Gamma of parameter $\beta/2$, since $(\gamma^{(n)}_k)_{k \geq 1}$ are i.i.d. 
The conditional probability of the event above is then bounded by the probability
that a Gamma variable of parameter $\beta/2$
is in an interval of length 
$\mathcal{O} (e^{-t} |z - Z^{(n)}_1|)$. 
For $\beta \geq 2$, the density of such a Gamma variable is bounded, so the probability to be in an interval of length $h$ is $\mathcal{O}_{\beta} (h)$. For $\beta < 2$, the density at $t$ 
is $(\Gamma(\beta/2))^{-1} t^{\beta/2 - 1}$, which gives a probability $\mathcal{O}_{\beta} (h^{\beta/2})$. 
Hence, in any case, the probability is dominated by $h^{\alpha}$ for some $\alpha \in (0,1]$ depending only on $\beta$, which gives 
$$\mathbb{P} \left( \left| \sum_{k=1}^n 
\frac{ \gamma_k^{(n)}}{ z - Z^{(n)}_k}  \right|
\leq e^{-t} \right) 
= \mathcal{O}_{\beta} ((e^{-t} |z - Z^{(n)}_1|)^{\alpha}) 
$$
Integrating for $t$ between $K_{4} \log n$ and infinity, we get 
$$\mathbb{E} \left[ \left(\log_- \left| \sum_{k=1}^n 
\frac{ \gamma_k^{(n)}}{ z - Z^{(n)}_k} \right| 
- K_{4} \log n  \right)_+ \right] 
= \mathcal{O}_{\beta} 
( e^{-  \alpha K_{4} \log n}   ( |Z^{(n)}_1| + |z| )^{\alpha}).$$
For $z \in \mathbb{D}_R$, we get, 
for $K_{4} := A/\alpha $, 
\begin{align*}
\mathbb{E} \left[ \left(\log_- \left| \sum_{k=1}^n 
\frac{ \gamma_k^{(n)}}{ z - Z^{(n)}_k} \right| 
- K_{4} \log n  \right)_+ \right] 
& = \mathcal{O}_{\beta}  
\left( n^{-A} ( (|Z^{(n)}_1| + R)^{\alpha} )  \right)
\\ & = \mathcal{O}_{\beta, R} \left( n^{-A}  (1 + |Z^{(n)}_1|)^{\alpha} \right). 
\end{align*}
Hence, 
$$
\int_{\mathbb{D}_R}
\mathbb{E} \left[ \left(\log_- \left| \sum_{k=1}^n 
\frac{ \gamma_k^{(n)}}{ z - Z^{(n)}_k} \right| 
- K_{4} \log n  \right)_+ \right] dm(z) 
 = \mathcal{O}_{\beta, R} \left( n^{-A} (1 + |Z^{(n)}_1|)^{\alpha} \right),$$
 which completes the proof of the proposition, since $\alpha \leq 1$.
\end{proof}
\begin{remark} \label{generaldistribution}
A careful look of the proof above shows that 
the proposition remains true as soon as the i.i.d., nonnegative random variables 
$(\gamma^{(n)}_k)_{n \geq 2, 1 \leq k \leq n}$ follow a distribution
with H\"older continuity and at least polynomially decaying tail, i.e.
$$  \sup_{R \geq 1} R^c \, \mathbb{P} (\gamma^{(2)}_1 > R )  < \infty$$
for some $c > 0$,
and
$$ \sup_{0 \leq a \leq b, \,  b-a \leq h} h^{-\alpha} \mathbb{P} (\gamma^{(2)}_1
\in [a,b]) < \infty $$
for some $\alpha \in (0,1]$. 
\end{remark}
We are now ready to prove Theorem \ref{thm:a.e.converg1stderivative}.
\begin{proof}[Proof of Theorem \ref{thm:a.e.converg1stderivative}]
It is enough to prove the following fact: 
the conditional probability that 
$(\nu_n)_{n \geq 1}$ converges to $\mu$ given $(Z^{(n)}_k)_{n \geq 1, 1 \leq k \le n} $ is almost surely equal to $1$. It is sufficient to show that this conditional probability is $1$ as soon as $(\mu_n)_{n \geq 1}$ tends to $\mu$, since this event, which depends only on $(Z^{(n)}_k)_{n \geq 1, 1 \leq k \le n} $,  is assumed to have probability $1$. 
Because of the conditioning, we can then suppose, without loss of generality, that $(Z^{(n)}_k)_{n \geq 1, 1 \leq k \le n} $  is deterministic.

Let $A  \geq 1$ be chosen in such a way that $\mu( \mathbb{D}_A) \geq 1/2$. 
We have 
$\mu_n (\mathbb{D}_{2A}) 
> 0$ for $n$ large enough, which implies that the point with smallest modulus in the support of $\mu_n$ has modulus at most $2A$. 
Hence, the sequence $(\min_{1 \leq k \leq n} |Z^{(n)}_k|)_{n \geq 1}$ is bounded, which implies, by applying Proposition \ref{prop1} for $A > 1$,
and Borel-Cantelli lemma, 
that for any smooth function $\varphi : \mathbb{C} \rightarrow \mathbb{R}$ with compact support,  almost surely, 
$$\left| \int_{\mathbb{C}} \varphi d \nu_n
- \int_{\mathbb{C}}\varphi d \mu_n  
\right| \leq \frac{K \log n}{n}$$
for all but finitely many values of $n$. In particular,   
$$ \int_{\mathbb{C}} \varphi d \nu_n
- \int_{\mathbb{C}}\varphi d \mu_n  
\underset{n \rightarrow \infty}{\longrightarrow} 0
$$
almost surely. 
From the convergence of $(\mu_n)_{n \geq 1}$ towards $\mu$, we deduce that almost surely, 
$$\int_{\mathbb{C}} \varphi d \nu_n 
\underset{n \rightarrow \infty}{\longrightarrow} \int_{\mathbb{C}} \varphi d \mu.$$
This convergence almost surely occurs for all functions $\varphi$ in a given countable family of smooth functions with compact support. 
This is sufficient to get vague  convergence of $(\nu_n)_{n \geq 1}$ towards $\mu$, and then weak convergence since all measures have 
total mass equal to $1$. 
\end{proof}
To prove Theorem \ref{thm:a.e.converg_n_logn_derivative_prop3}, we observe that the coefficients $(\gamma^{(n,j)}_k)_{1 \leq k \leq n+1-j}$ are independent of 
$ (Z_{\ell}^{(n,j-1)})_{1 \leq \ell \leq n+1-j}$, and then we are in the same situation as above when we go
from the points $ (Z_{\ell}^{(n,j-1)})_{1 \leq \ell \leq n+1-j}$ to the points $(Z_k^{(n,j)})_{1 \leq k \leq n-j}$. We can then apply Proposition \ref{prop1} here. 

\begin{proof}[Proof of Theorem \ref{thm:a.e.converg_n_logn_derivative_prop3}]
    Let $\varphi : \mathbb{C} \rightarrow \mathbb{R}$ be smooth with compact support, and let $\varepsilon> 0$. 
We have
\begin{align*}
    \mathbb{P} \left( \left| \int_{\mathbb{C}} \varphi d \mu_{n, m(n)} - \int_{\mathbb{C}} \varphi d \mu_{n, 0} 
\right| \geq \varepsilon \right) & \leq
\sum_{j=1}^{m(n)} \mathbb{P} \left( \left| \int_{\mathbb{C}} \varphi d \mu_{n, j-1} - \int_{\mathbb{C}} \varphi d \mu_{n, j} 
\right| \geq \frac{\varepsilon}{ m(n)} \right) 
\end{align*}
For $n$ larger than some integer $n_0 \geq 2$ depending only on $\varepsilon$, $\beta$, $A$, $\varphi$ and the sequence $(m(n))_{n \geq 2}$, 
we have 
$$\frac{\varepsilon}{ m(n)} \geq \frac{K \log n}{n-m(n)},$$
where $K$ is defined in Proposition \ref{prop1}. This is due to the fact that by assumption, $m(n) (\log n)/n$ tends to zero when $n \rightarrow \infty$, and 
 $n-m(n)$ is equivalent to $n$ since $m(n)/n \rightarrow 0$. 
We deduce that for $n \geq n_0$, 
$$\mathbb{P} \left( \left| \int_{\mathbb{C}} \varphi d \mu_{n, m(n)} - \int_{\mathbb{C}} \varphi d \mu_{n, 0} 
\right| \geq \varepsilon \right)
\leq \sum_{j=1}^{m(n)} \mathbb{P} \left( \left| \int_{\mathbb{C}} \varphi d \mu_{n, j-1} - \int_{\mathbb{C}} \varphi d \mu_{n, j} 
\right| \geq \frac{K \log n}{n-m(n)} \right) 
$$

 We can apply Proposition \ref{prop1} to the 
 $n+1-j$ points $ (Z_{k}^{(n,j-1)})_{1 \leq k \leq n+1-j}$
and the independent weights $(\gamma^{(n,j)}_k)_{1 \leq k \leq n+1-j}$. 
 Since $K \log n / (n-m(n)) \geq K \log (n+1-j) / (n+1-j)$ for $1 \leq j \leq m(n)$, 
 we get 
 $$\mathbb{P} \left( \left| \int_{\mathbb{C}} \varphi d \mu_{n, m(n)} - \int_{\mathbb{C}} \varphi d \mu_{n, 0} 
\right| \geq \varepsilon \right)
\leq \sum_{j=1}^{m(n)} C (n+1-j)^{-A} 
\left( 1+ \mathbb{E} \left[ \min_{1 \leq k \leq n+1-j}
|Z_{k}^{(n,j-1)}| \right] \right),
$$
for some $C$ depending only on $\beta, A, \varphi$. 
By assumption, 
$$\left( 1+ \mathbb{E} \left[ \min_{1 \leq k \leq n+1-j}
|Z_{k}^{(n,j-1)}| \right] \right) \leq C' n^B$$
for some $C' > 0$: the quantity $C'$ depends only
on $B$, $\beta$ and the distribution of the initial points $(Z_k^{(n,0)})_{1 \leq k \leq n}$. 
For $n$ larger that some integer $n_1 \geq 2$ depending only on the sequence $(m(n))_{n \geq 2}$, we have 
$n+1 -j \geq 9n/10$ for $1 \leq j \leq m(n)$, since 
$m(n)/n \rightarrow 0$. 
Hence, for $C'' > 0$ depending only on $\beta$, $A$, $B$, $\varphi$, and the distribution of $(Z_k^{(n,0)})_{1 \leq k \leq n}$, and for $n \geq \max(n_0, n_1)$, 
 $$\mathbb{P} \left( \left| \int_{\mathbb{C}} \varphi d \mu_{n, m(n)} - \int_{\mathbb{C}} \varphi d \mu_{n, 0} 
\right| \geq \varepsilon \right) \leq 
C'' m(n) n^{-A} n^{B} \leq C'' n^{B-A+1}.$$
By choosing $A = B + 3$, we deduce, since $C''$ does not depend on $n$, that 
$$\sum_{n \geq \max (n_0, n_1)} 
\mathbb{P} \left( \left| \int_{\mathbb{C}} \varphi d \mu_{n, m(n)} - \int_{\mathbb{C}} \varphi d \mu_{n, 0} 
\right| \geq \varepsilon \right) < \infty.$$
Taking $\varepsilon = 1/q$ and using Borel-Cantelli lemma, we deduce that  almost surely, for all 
integers $q \geq 1$, 
$$\left| \int_{\mathbb{C}} \varphi d \mu_{n, m(n)} - \int_{\mathbb{C}} \varphi d \mu_{n, 0} 
\right| <1/q$$
for all but finitely many $n \geq 2$. 
Hence, almost surely, 
$$\int_{\mathbb{C}} \varphi d \mu_{n, m(n)} - \int_{\mathbb{C}} \varphi d \mu_{n, 0} 
\underset{n \rightarrow \infty}{\longrightarrow} 
0.$$
From the convergence of $(\mu_{n,0})_{n \geq 1}$ towards $\mu$, we deduce that almost surely, 
$$\int_{\mathbb{C}} \varphi d  \mu_{n, m(n)}
\underset{n \rightarrow \infty}{\longrightarrow} \int_{\mathbb{C}} \varphi d \mu.$$
This convergence almost surely occurs for all functions $\varphi$ in a given countable family of smooth functions with compact support. 
This is sufficient to get vague  convergence of $(\mu_{n,m(n)})_{n \geq 1}$ towards $\mu$, and then weak convergence since all measures have 
total mass equal to $1$. 

\end{proof} 

\begin{proof} [Proof of Lemma \ref{lem:prop4}]
For $1 \leq j \leq m(n)$, and $1 \leq k \leq n$ 
such that $Z^{(n,j)}_{k}$ is not one of the points 
$(Z^{(n,j-1)}_{\ell})_{1 \leq \ell \leq n+1-j}$, we have 
$$\sum_{\ell = 1}^{n+1-j} \frac{ \gamma^{(n,j)}_{\ell}}{ Z^{(n,j)}_k - Z^{(n,j-1)}_{\ell}} = 0, $$
$$\sum_{\ell = 1}^{n+1-j} \frac{ \gamma^{(n,j)}_{\ell} \left(\overline{Z^{(n,j)}_k - Z^{(n,j-1)}_{\ell} }\right) }
{ |Z^{(n,j)}_k - Z^{(n,j-1)}_{\ell}|^2} = 0, $$
$$\left(\sum_{\ell = 1}^{n+1-j} \frac{ \gamma^{(n,j)}_{\ell}}{|Z^{(n,j)}_k - Z^{(n,j-1)}_{\ell}|^2}  \right) 
\overline{Z^{(n,j)}_k} 
=\sum_{\ell = 1}^{n+1-j}
\frac{ \gamma^{(n,j)}_{\ell}}{|Z^{(n,j)}_k - Z^{(n,j-1)}_{\ell}|^2}
\overline{Z^{(n,j-1)}_\ell},
$$
Hence, the points $(Z^{(n,j)}_{k})_{1 \leq k \leq n}$ are barycenters of 
the points $(Z^{(n,j-1)}_{\ell})_{1 \leq \ell \leq n+1-j}$, which implies that 
$$\max_{1 \leq k \leq n-j} |Z_k^{(n,j)}| \leq 
\max_{1 \leq \ell \leq n+1-j} |Z_\ell^{(n,j-1)}|,$$
and by induction, 
$$\max_{1 \leq k \leq n-j} |Z_k^{(n,j)}|
\leq \max_{1 \leq \ell \leq n} |Z_\ell^{(n,0)}|.$$
Hence, 
\begin{align*} \sup_{n \geq 2} n^{-B} \sup_{ 0 \leq j \leq m(n)-1} \mathbb{E} \left[ \min_{1 \leq k \leq n-j} |Z_k^{(n,j)}| \right]  
& \leq  \sup_{n \geq 2} n^{-B} \sup_{ 0 \leq j \leq m(n)-1} \mathbb{E} \left[ \max_{1 \leq k \leq n-j} |Z_k^{(n,j)}| \right]  
\\ & = \sup_{n \geq 2} n^{-B}
\mathbb{E} \left[ \max_{1 \leq k \leq n} |Z_k^{(n,0)}| \right]
< \infty.
\end{align*}

 \end{proof}

\begin{proof}[Proof of Theorem \ref{thm:a.e.converg_n_logn_derivative_prop5}]
We first condition on $(Z^{(n,0)}_k)_{1 \leq k \leq n}$, and we check that 
the assumptions of Theorem \ref{thm:a.e.converg_n_logn_derivative_prop3} are almost surely satisfied. 
The almost sure convergence of $\mu_{n,0}$ to $\mu$ is a direct application of the law of large numbers. 
It remains to check that the last assumption of Theorem \ref{thm:a.e.converg_n_logn_derivative_prop3} is almost surely satisfied, which is implied by the fact that for some $B > 0$, almost surely, 
$$\sup_{n \geq 2} n^{-B} 
\max_{1 \leq k \leq n} |Z_k^{(n,0)}| < \infty.$$
Notice that our initial conditioning 
on $(Z^{(n,0)}_k)_{1 \leq k \leq n}$ allows us to remove the expectation in the assumption of Lemma \ref{lem:prop4}. 
By Borel-Cantelli lemma, it is enough to check that 
$$\mathbb{P} \left(
\max_{1 \leq k \leq n} |Z_k^{(n,0)}| > n^B \right) 
= \mathcal{O}(n^{-2}), 
$$
and then, using a union bound, 
$$\mathbb{P} ( |Z| > n^B) 
= \mathcal{O} (n^{-3})$$
where $Z$ is a random variable of distribution $\mu$. Now, 
by assumption on $\mu$, 
$$\mathbb{P} (|Z| > n^B) 
= \mathcal{O}( n^{-Bc}),$$
so the bound we need is satisfied when $B \geq 3/c$. 

 \end{proof} 

\section{The circular variant of the randomized derivative model}\label{sec:variation_randomized} 
We prove the following: 
\begin{proposition} \label{prop6}
For all $\beta > 0$, $A > 0$, and for all smooth functions $\varphi : \mathbb{C} \mapsto \mathbb{R}$, 
with compact support, there exists 
$C, K > 0$, depending only on 
$\beta$, $A$, $\varphi$, such that 
for all $n \geq 2$, 
$$ \mathbb{P} \left( \left|\int_{\mathbb{C}} \varphi d \tilde\nu_n
- \int_{\mathbb{C}}\varphi d \mu_n
\right| 
\geq \frac{K \log n }{n} \right) 
\leq C n^{-A} \left(1 + \mathbb{E} \left[
\min_{1 \leq k \leq n} |Z_k^{(n)}| \right]
\right). 
$$
\end{proposition}
\begin{proof}
By first conditioning on $(Z^{(n)}_k)_{1 \leq k \leq n}$, which does not change the distribution of 
$(\gamma^{(n)}_k)_{1 \leq k \leq n}$ since these families of variables are assumed to be independent, 
we can assume that $(Z^{(n)}_k)_{1 \leq k \leq n}$ is deterministic. 
Moreover, the ordering of
$(Z^{(n)}_k)_{1 \leq k \leq n}$ does not change the distribution of the random measure $\nu_n$, since $(\gamma^{(n)}_k)_{1 \leq k \leq n}$
are i.i.d.: we can assume that 
$|Z^{(n)}_1| \leq |Z^{(n)}_k| $
for $2 \leq k \leq n$. It is then enough to prove, for 
deterministic $(Z^{(n)}_k)_{1 \leq k \leq n}$, 
$$ \mathbb{P} \left( \left|\int_{\mathbb{C}} \varphi d \tilde\nu_n
- \int_{\mathbb{C}}\varphi d \mu_n \right| 
\geq \frac{K \log n }{n} \right) 
\leq C n^{-A} \left(1 +
 |Z_1^{(n)}| 
\right). 
$$

 We have 
$$n \int_{\mathbb{C}} \varphi d \mu_n 
= \frac{1}{2 \pi} \int_{\mathbb{C}} \log |P_n | \, \Delta \varphi \, d m (z)$$
and 
$$n \int_{\mathbb{C}} \varphi d \tilde\nu_n 
= \frac{1}{2 \pi} \int_{\mathbb{C}} \log |Q_n | \, \Delta \varphi \, d m (z)$$
where $m$ is the Lebesgue measure on $\mathbb{C}$. 
Hence, 

$$\int_{\mathbb{C}} \varphi d \tilde\nu_n
- \int_{\mathbb{C}}\varphi d \mu_n 
=     \frac{1}{2 \pi n} 
\int_{\mathbb{C}} \log \left|  \sum_{k=1}^n  \frac{Z^{(n)}_k + z}{Z^{(n)}_k - z} \gamma^{(n)}_k\right| \, \Delta \varphi \, d m (z).$$

The right-hand side is bounded by 
the maximum of $|\Delta \varphi|$, 
multiplied by

$$\frac{1}{2 \pi n} 
\int_{\mathbb{D}_R} 
\log_+ \left|   \sum_{k=1}^n  \frac{Z^{(n)}_k + z}{Z^{(n)}_k - z} \gamma^{(n)}_k \right| 
d m(z) 
+ \frac{1}{2 \pi n} 
\int_{\mathbb{D}_R} 
\log_- \left|  \sum_{k=1}^n  \frac{Z^{(n)}_k + z}{Z^{(n)}_k - z} \gamma^{(n)}_k\right|  d m(z), 
$$
where $R \geq 1$ is chosen in such a way that 
$\varphi$ is supported on 
$$\mathbb{D}_R := \{ z \in \mathbb{C}, |z| \leq R\}. $$

It is then enough to show that there exists $C, K_2, K_3 > 0$, depending only on $\beta$, $A$ and $R$, such that

$$\mathbb{P} \left(
\int_{\mathbb{D}_R}
\log_+ \left|\sum_{k=1}^n  \frac{Z^{(n)}_k + z}{Z^{(n)}_k - z} \gamma^{(n)}_k \right| 
d m(z)  > K_2 \log n \right) 
\leq C  n^{-A} (1 + |Z^{(n)}_1|) 
$$
and 
$$\mathbb{P} \left(
\int_{\mathbb{D}_R}
\log_- \left| \sum_{k=1}^n  \frac{Z^{(n)}_k + z}{Z^{(n)}_k - z} \gamma^{(n)}_k\right| 
d m(z)  > K_3 \log n \right) 
\leq C  n^{-A} (1 + |Z^{(n)}_1|) 
$$

For the integral involving $\log_+$,
we first observe, as before, that   
$$\mathbb{P} \left( \underset{1 \leq k \leq n}{\max} \gamma^{(n)}_k 
\geq n^{(A+1)/c} \right) \leq  C n^{-A} $$
for some $C$ depending only on $\beta$ and $A$ and $c > 0$ universal.  
It is then enough to show that under the complementary event $\gamma_k^{(n)} \leq n^{(A+1)/c}$, we have an equality of the form

$$\int_{\mathbb{D}_R}
\log_+ \left| \sum_{k=1}^n  \frac{Z^{(n)}_k + z}{Z^{(n)}_k - z} \gamma^{(n)}_k \right| dm(z) \leq K_2 \log n.  $$

Now, for any integer $m \geq 0$, the term  
$\log_+$ just above
can be larger than 
$(10  + (A+1)/c) \log n + m$ only if one of the 
terms 
$$ \frac{Z^{(n)}_k + z}{Z^{(n)}_k - z} \gamma^{(n)}_k 
= \left(\frac{2 Z^{(n)}_k }{Z^{(n)}_k - z} - 1\right) 
\gamma^{(n)}_k $$
has modulus larger than or equal to $n^{9 + (A+1)/c} e^m$,
which implies, for $\gamma_k^{(n)} \leq n^{(A+1)/c}$, 
$$\left| \frac{2 Z^{(n)}_k }{Z^{(n)}_k - z} \right| 
\geq n^9 e^m - 1 \geq n^9 e^{m} (1 - 2^{-9}),
$$
$$|Z^{(n)}_k - z| \leq \frac{2 Z^{(n)}_k}{n^9 e^{m} (1 - 2^{-9})},$$
which restricts $z$ to a disc of radius at most  $3 e^{-m} n^{-9}|Z^{(n)}_k|$. 
Moreover, if $|Z^{(n)}_k| \geq 2R$, the 
inequality above cannot be satisfied for $z \in \mathbb{D}_R$, since in this case, $|Z_{k}^{(n)} - z| \geq |Z_{k}^{(n)}|/2$ and then 
$$
\left| \frac{2 Z^{(n)}_k }{Z^{(n)}_k - z} \right|  
\leq 4 <  n^9 e^m (1 - 2^{-9}).$$
The set of $z \in \mathbb{D}_R$
for  which   $\log_+$ is larger than $(10 + (A+1)/c) \log n + m$ is then included in the union of at most $n$
discs of radius $3 e^{-m} n^{-9}|Z^{(n)}_k|$
for $|Z^{(n)}_k| \leq 2R$. The total area of these discs
is dominated by $n e^{-2m} n^{-18} R^2$.

Bounding $\log_+$ by $(10 + (A+1)/c) \log n$ 
if it is smaller than $(10 + (A+1)/c) \log n$, and by 
 $(10 + (A+1)/c) \log n + m + 1$ if it is between 
 $(10 + (A+1)/c) \log n  + m$ and $(10 + (A+1)/c) \log n  + m+1$,
 we deduce that 

\begin{align*}
&  
\int_{\mathbb{D}_R}
\log_+ \left|  \sum_{k=1}^n  \frac{Z^{(n)}_k + z}{Z^{(n)}_k - z} \gamma^{(n)}_k \right| 
d m(z) 
\\ & \leq  
( \pi R^2) (10 + (A+1)/c) \log n
+  \sum_{m=0}^{\infty}
 \mathcal{O} (e^{-2m} n^{-17} R^2) 
 ( (10 + (A+1)/c) \log n + m + 1)
\leq  K_2 \log n 
\end{align*}

where $K_2 > 0$ depends only on $A$ and $R$. 

It remains to prove the bound on the integral involving $\log_-$. 
If we replace $\log_-$ by its minimum with $K_4 \log n$
for some $K_4 > 0$ depending only on $\beta$, $A$, $R$, we get an integral on $\mathbb{D}_{R} $
bounded by $\pi R^2 K_4 \log n$, which shows that it is enough to check

$$\mathbb{P} \left(
\int_{\mathbb{D}_R}
\left(\log_- \left| \sum_{k=1}^n  \frac{Z^{(n)}_k + z}{Z^{(n)}_k - z} \gamma^{(n)}_k  \right|  - K_4 \log n \right)_+
d m(z)  \geq  10 \log n  \right) 
\leq C  n^{-A} (1 + |Z^{(n)}_1|) 
$$
for $C > 0$ depending only on $\beta, A, R$. 
Since $10 \log n \geq 1$, by Markov inequality, it is enough to show 
$$ \int_{\mathbb{D}_R}
\mathbb{E} \left[ \left(\log_- \left| \sum_{k=1}^n  \frac{Z^{(n)}_k + z}{Z^{(n)}_k - z} \gamma^{(n)}_k  \right| 
- K_{4} \log n  \right)_+ \right] dm(z) 
 \leq C n^{-A} ( 1 + |Z^{(n)}_1|) .$$
 
The last expectation is equal to 

$$\int_{K_{4} \log n}^{\infty} 
\mathbb{P} \left( \left|\sum_{k=1}^n  \frac{Z^{(n)}_k + z}{Z^{(n)}_k - z} \gamma^{(n)}_k   \right|
\leq e^{-t} \right)  dt.
$$

We recall that we assume $(Z^{(n)}_k)_{ 
1 \leq k \leq n}$ deterministic. 
If we condition on $(\gamma^{(n)}_k)_{k \geq 2}$, 
the conditional distribution of $\gamma^{(n)}_1$ remains Gamma of parameter $\beta/2$, since $(\gamma^{(n)}_k)_{k \geq 1}$ are i.i.d. 
The conditional probability of the event above is then bounded by the probability
that a Gamma variable of parameter $\beta/2$
is in an interval of length 
$\mathcal{O} (e^{-t} |z - Z^{(n)}_1||z + Z^{(n)}_1|^{-1})$.

For $\beta \geq 2$, the density of such a Gamma variable is bounded, so the probability to be in an interval of length $h$ is $\mathcal{O}_{\beta} (h)$. For $\beta < 2$, the density at $t$ 
is at most $(\Gamma(\beta/2))^{-1} t^{\beta/2 - 1}$, which gives a probability $\mathcal{O}_{\beta} (h^{\beta/2})$. 
Hence, 

$$\mathbb{P} \left( \left|\sum_{k=1}^n  \frac{Z^{(n)}_k + z}{Z^{(n)}_k - z} \gamma^{(n)}_k  \right|
\leq e^{-t} \right) 
=  \mathcal{O}_{\beta}\left( \left(e^{-t} |z - Z^{(n)}_1||z + Z^{(n)}_1|^{-1}\right)^{\min(1, \beta/2)} \right), 
$$
and 
$$\int_{K_{4} \log n}^{\infty} 
\mathbb{P} \left( \left|\sum_{k=1}^n  \frac{Z^{(n)}_k + z}{Z^{(n)}_k - z} \gamma^{(n)}_k  \right|
\leq e^{-t} \right) dt
= \mathcal{O}_{\beta} \left(\left(n^{-K_4} |z - Z^{(n)}_1||z + Z^{(n)}_1|^{-1}\right)^{\min(1, \beta/2)}
\right). 
$$
For $z \in \mathbb{D}_R$, we deduce 
$$\int_{K_{4} \log n}^{\infty} 
\mathbb{P} \left( \left|\sum_{k=1}^n  \frac{Z^{(n)}_k + z}{Z^{(n)}_k - z} \gamma^{(n)}_k  \right|
\leq e^{-t} \right) dt
= \mathcal{O}_{\beta, R} \left(\left(n^{-K_4} (1 + |Z^{(n)}_1|)|z + Z^{(n)}_1|^{-1}\right)^{\min(1, \beta/2)}
\right). $$
Integrating in $z$, we deduce 
\begin{align*}
& \int_{\mathbb{D}_R}
\mathbb{E} \left[ \left(\log_- \left| \sum_{k=1}^n  \frac{Z^{(n)}_k + z}{Z^{(n)}_k - z} \gamma^{(n)}_k\right| 
- K_{4} \log n  \right)_+ \right] dm(z)
\\ &  =  \mathcal{O}_{R, \beta}
\left( n^{-K_4 \min(1, \beta/2)} (1 + |Z_1^{(n)})
\int_{\mathbb{D}_R} |z + Z_1^{(n)}|^{-\min(1, \beta/2)} 
dm (z)\right).
 \end{align*}
We have 
$$\int_{\mathbb{D}_R} |z + Z_1^{(n)}|^{-\min(1, \beta/2)} 
\mathds{1}_{|z + Z_1^{(n)}| \geq 1} dm (z) \leq \pi R^2$$
since the quantity to integrate is bounded by $1$.
On the other hand
\begin{align*}\int_{\mathbb{D}_R} |z + Z_1^{(n)}|^{-\min(1, \beta/2)} 
\mathds{1}_{|z + Z_1^{(n)}| < 1} dm (z)
& \leq \int_{\mathbb{D}_1} |z|^{-\min(1, \beta/2)} dm(z) 
\\ & = \int_0^1 \int_0^{2 \pi} r^{-\min(1, \beta/2)} r dr d\theta  \\ & \leq 2 \pi.
\end{align*}
By choosing $K_4 \geq A/ \min(1, \beta/2)$, we deduce
$$
\int_{\mathbb{D}_R}
\mathbb{E} \left[ \left(\log_- \left| \sum_{k=1}^n  \frac{Z^{(n)}_k + z}{Z^{(n)}_k - z} \gamma^{(n)}_k\right| 
- K_{4} \log n  \right)_+ \right] dm(z) 
 = \mathcal{O}_{\beta, R} \left( n^{-A} (1 + |Z^{(n)}_1|) \right)$$
which completes the proof
\end{proof}

\begin{proof}[Proof of Theorem \ref{thm:prop7}]
We deduce this theorem exactly in the same way as we deduced
Theorem \ref{thm:a.e.converg1stderivative} from Proposition \ref{prop1}. 
\end{proof}

\begin{proof}[Proof of Theorem \ref{thm:prop8}]
    Let $\varphi : \mathbb{C} \rightarrow \mathbb{R}$ be smooth with compact support, and let $\varepsilon> 0$. 
We have
\begin{align*}
    \mathbb{P} \left( \left| \int_{\mathbb{C}} \varphi d \tilde\mu_{n, m(n)} - \int_{\mathbb{C}} \varphi d \tilde\mu_{n, 0} 
\right| \geq \varepsilon \right) & \leq
\sum_{j=1}^{m(n)} \mathbb{P} \left( \left| \int_{\mathbb{C}} \varphi d \tilde\mu_{n, j-1} - \int_{\mathbb{C}} \varphi d \tilde\mu_{n, j} 
\right| \geq \frac{\varepsilon}{ m(n)} \right) 
\end{align*}
For $n\ge n_0 \geq 2$ depending only on $\varepsilon$, $\beta$, $A$, $\varphi$ and the sequence $(m(n))_{n \geq 2}$, 
we have
$$\frac{\varepsilon}{ m(n)} \geq \frac{K \log n}{n},$$
where $K$ is defined in Proposition \ref{prop6}: indeed, $m(n) (\log n)/n$ tends to zero when $n \rightarrow \infty$. 
We deduce that for $n \geq n_0$, 
$$\mathbb{P} \left( \left| \int_{\mathbb{C}} \varphi d \tilde\mu_{n, m(n)} - \int_{\mathbb{C}} \varphi d \tilde\mu_{n, 0} 
\right| \geq \varepsilon \right)
\leq \sum_{j=1}^{m(n)} \mathbb{P} \left( \left| \int_{\mathbb{C}} \varphi d \tilde\mu_{n, j-1} - \int_{\mathbb{C}} \varphi d \tilde\mu_{n, j} 
\right| \geq \frac{K \log n}{n} \right) 
$$
 We can apply Proposition \ref{prop6} to the 
 $n$ points $ (Z_{k}^{(n,j-1)})_{1 \leq k \leq n}$
and the independent weights $(\gamma^{(n,j)}_k)_{1 \leq k \leq n}$. 
We get 
 $$\mathbb{P} \left( \left| \int_{\mathbb{C}} \varphi d \tilde\mu_{n, m(n)} - \int_{\mathbb{C}} \varphi d \tilde\mu_{n, 0} 
\right| \geq \varepsilon \right)
\leq \sum_{j=1}^{m(n)} C n^{-A} 
\left( 1+ \mathbb{E} \left[ \min_{1 \leq k \leq n}
|Z_{k}^{(n,j-1)}| \right] \right),
$$
for some $C$ depending only on $\beta, A, \varphi$. 
By assumption, 
$$\left( 1+ \mathbb{E} \left[ \min_{1 \leq k \leq n}
|Z_{k}^{(n,j-1)}| \right] \right) \leq C' n^B$$
for some $C' > 0$ depending only 
on $B$, $\beta$ and the distribution of the initial points $(Z_k^{(n,0)})_{1 \leq k \leq n}$. 
Hence, for $n \geq \max(n_0, n_1)$, 
$$\mathbb{P} \left( \left| \int_{\mathbb{C}} \varphi d \tilde\mu_{n, m(n)} - \int_{\mathbb{C}} \varphi d \tilde\mu_{n, 0} 
\right| \geq \varepsilon \right) \leq 
C' m(n) n^{-A} n^{B} \leq C' n^{B-A+1}.$$
By choosing $A \ge B + 3$, we deduce the convergence of $(\tilde\mu_{n,m(n)})_{n \geq 1}$ towards $\mu$, by using Borel-Cantelli lemma. 

\end{proof}

\begin{lemma}\label{lem:prop9}
If for some $B > 0$
$$ \sup_{n \geq 2} n^{-B}   \mathbb{E} \left[ \max_{1 \leq k \leq n} |Z_k^{(n,0)}| \right] < \infty, 
$$
then last condition of Theorem \ref{thm:prop8}
is satisfied. 
 \end{lemma}
 \begin{proof} 
 Observe that
 \begin{equation}
     2\,\text{Re}\,\left(\frac{z+Z^{(n,j-1)}}{z-Z^{(n,j-1)}}\right)=\frac{z+Z^{(n,j-1)}}{z-Z^{(n,j-1)}}+\frac{\bar{z}+\overline{Z^{(n,j-1)}}}{\bar{z}-\overline{Z^{(n,j-1)}}}=\frac{2(|z|^2-|Z^{(n,j-1)}|^2)}{(\bar{z}-\overline{Z^{(n,j-1)}})(z-Z^{(n,j-1)})}.
 \end{equation}
Therefore, for $|z|> \max\limits_{1\leq \ell \leq n}|Z^{(n,j-1)}_{\ell}|$, we have
$$\text{Re}\,\left(\sum\limits_{k=1}^{n}\frac{z+Z_{k}^{(n,j-1)}}{z-Z_{k}^{(n,j-1)}}\gamma^{(n,j)}_k \right) >0.$$ Hence, for $1 \leq j \leq m(n)$, and $1 \leq k \leq n$, we have 
\begin{equation} 
\max_{1 \leq k \leq n} |Z_k^{(n,j)}| \leq 
\max_{1 \leq \ell \leq n} |Z_\ell^{(n,j-1)}|,
\end{equation}
Hence, 
\begin{align*} \sup_{n \geq 2} n^{-B} \sup_{ 0 \leq j \leq m(n)-1} \mathbb{E} \left[ \min_{1 \leq k \leq n} |Z_k^{(n,j)}| \right]  
& \leq  \sup_{n \geq 2} n^{-B} \sup_{ 0 \leq j \leq m(n)-1} \mathbb{E} \left[ \max_{1 \leq k \leq n} |Z_k^{(n,j)}| \right]  
\\ & = \sup_{n \geq 2} n^{-B}
\mathbb{E} \left[ \max_{1 \leq k \leq n} |Z_k^{(n,0)}| \right]
< \infty.
\end{align*}
\end{proof}

\begin{proof}[Proof of Theorem \ref{thm:a.e.converg_n_logn_derivative_prop9}]
We deduce this theorem from Theorem \ref{thm:prop8} and Lemma 
\ref{lem:prop9}, 
exactly in the same way as 
 Theorem \ref{thm:a.e.converg_n_logn_derivative_prop5} is deduced from Theorem
\ref{thm:a.e.converg_n_logn_derivative_prop3} and Lemma \ref{lem:prop4}
 \end{proof}

\section{Conclusion and Future Directions}
We believe that the conclusion of Theorem \ref{thm:a.e.converg_n_logn_derivative_prop3}, Theorem \ref{thm:a.e.converg_n_logn_derivative_prop5}, Theorem \ref{thm:prop8} and Theorem \ref{thm:a.e.converg_n_logn_derivative_prop9} are suboptimal, and that in fact the same statement should hold with order $o(n)$ of randomized derivatives. Moreover, although we made important progress with $o(\frac{n}{\log n})$, there is still room for improvement.
It would be great if one can make a connection between our results  with randomized derivatives and
the usual derivatives case which corresponds to $\beta =\infty$ in our setting. We also refer to \cite{michelen2022zeros} and \cite{michelen2023almost} for related conjectures.

\textbf{Acknowledgments.} 
AG thanks the European ERC 101054746 grant
ELISA for its support
(indirect costs).
TV is supported in part by NSF grant DMS-2005493. He also thanks Jonas Jalowy and Zakhar Kabluchko for useful discussions. The authors would like to thank the anonymous referee for carefully
reading the earlier version of this manuscript and provided valuable suggestions.


\end{document}